\documentclass[letterpaper, 10 pt, conference]{ieeeconf}  
\IEEEoverridecommandlockouts  
\usepackage{cite}
\usepackage{amsmath,amssymb,amsfonts}
\usepackage{algorithmic}
\usepackage[pdftex]{graphicx}
\usepackage{textcomp}
\usepackage{epsfig} 
\usepackage{comment}
\usepackage{color}
\usepackage{cases}
\usepackage{ascmac}
\usepackage[subrefformat=parens]{subcaption}
\usepackage{cuted}
\usepackage{enumerate}
\usepackage{algorithmic}
\usepackage{algorithm}
\usepackage{mathtools}
\usepackage{mathrsfs}

\usepackage{bm}

\usepackage{algorithm}	\usepackage{textcomp}
\usepackage{mathtools}	\usepackage{mleftright}
\usepackage{mathrsfs}	\usepackage{breqn}
\usepackage{autobreak}
\usepackage{bm}	
\allowdisplaybreaks[1]

\newtheorem{proposition}{Proposition}

\newtheorem{theorem}{Theorem}

\DeclareMathOperator{\tr}{tr}
\newcommand{\nc}{\newcommand}
\nc{\disp}{\displaystyle}
\nc{\argmax}{\mathop{\rm arg~max}\limits}
\nc{\argmin}{\mathop{\rm arg~min}\limits}
\nc{\lr}[1]{\ensuremath{\left(#1\right)}}
\nc{\mlr}[1]{\mleft(#1\mright)}
\nc{\alr}[1]{\ensuremath{\left\langle#1\right\rangle}} 
\nc{\blr}[1]{\ensuremath{\left[#1\right]}} 
\nc{\clr}[1]{\ensuremath{\left\{#1\right\}}} 
\nc{\nlr}[1]{\ensuremath{\left\|#1\right\|}}
\nc{\vlr}[1]{\ensuremath{\left|#1\right|}}
\nc{\amlr}[1]{\mleft\langle#1\mright\rangle} 
\nc{\bmlr}[1]{\mleft[#1\mright]} 
\nc{\cmlr}[1]{\mleft\{#1\mright\}} 
\nc{\nmlr}[1]{\mleft\|#1\mright\|}
\nc{\fnmlr}[1]{\mleft\|#1\mright\|_\mathcal{F}}
\nc{\vmlr}[1]{\mleft|#1\mright|}
\nc{\ffnmlr}[1]{\mleft\|#1\mright\|_F}

\nc{\rmvec}{\mathop{\rm vec}}
\nc{\rmtr}{\mathop{\rm tr}}
\nc{\tp}{\mathsf{T}}
\nc{\ex}{\mathbb{E}}
\nc{\normal}{\mathcal{N}}
\nc{\defeq}{:=}
\nc{\cpad}{~}
\nc{\setdef}[2]{\clr{#1 \mathrel{}\middle|\mathrel{} #2}}
\nc{\Z}{\mathbb{Z}}
\nc{\R}{\mathbb{R}}
\nc{\definite}[1]{\mathbb{S}_{#1}^{++}}
\nc{\semidefinite}[1]{\mathbb{S}_{#1}^{+}}
\nc{\N}{\mathbb{N}}
\nc{\ra}{\rightarrow}

\nc{\calA}{{\mathcal A}}    
\nc{\calB}{{\mathcal B}}    
\nc{\calC}{{\mathcal C}}    
\nc{\calD}{{\mathcal D}}    
\nc{\calE}{{\mathcal E}}    
\nc{\calF}{{\mathcal F}}    
\nc{\calG}{{\mathcal G}}    
\nc{\calH}{{\mathcal H}}    
\nc{\calI}{{\mathcal I}}    
\nc{\calJ}{{\mathcal J}}    
\nc{\calK}{{\mathcal K}}    
\nc{\calL}{{\mathcal L}}    
\nc{\calM}{{\mathcal M}}    
\nc{\calN}{{\mathcal N}}    
\nc{\calO}{{\mathcal O}}    
\nc{\calP}{{\mathcal P}}    
\nc{\calQ}{{\mathcal Q}}    
\nc{\calR}{{\mathcal R}}    
\nc{\calS}{{\mathcal S}}    
\nc{\calT}{{\mathcal T}}    
\nc{\calU}{{\mathcal U}}    
\nc{\calV}{{\mathcal V}}    
\nc{\calW}{{\mathcal W}}    
\nc{\calX}{{\mathcal X}}    
\nc{\calY}{{\mathcal Y}}    
\nc{\calZ}{{\mathcal Z}}    

\nc{\scrA}{{\mathscr A}}    
\nc{\scrB}{{\mathscr B}}    
\nc{\scrC}{{\mathscr C}}    
\nc{\scrD}{{\mathscr D}}    
\nc{\scrE}{{\mathscr F}}    
\nc{\scrF}{{\mathscr F}}    
\nc{\scrG}{{\mathscr G}}    
\nc{\scrH}{{\mathscr H}}    
\nc{\scrI}{{\mathscr I}}    
\nc{\scrJ}{{\mathscr J}}    
\nc{\scrK}{{\mathscr K}}    
\nc{\scrL}{{\mathscr L}}    
\nc{\scrM}{{\mathscr M}}    
\nc{\scrN}{{\mathscr N}}    
\nc{\scrO}{{\mathscr O}}    
\nc{\scrP}{{\mathscr P}}    
\nc{\scrQ}{{\mathscr Q}}    
\nc{\scrR}{{\mathscr R}}    
\nc{\scrS}{{\mathscr S}}    
\nc{\scrT}{{\mathscr T}}    
\nc{\scrU}{{\mathscr U}}    
\nc{\scrV}{{\mathscr V}}    
\nc{\scrW}{{\mathscr W}}    
\nc{\scrX}{{\mathscr X}}    
\nc{\scrY}{{\mathscr Y}}    
\nc{\scrZ}{{\mathscr Z}}    

\nc{\bbA}{{\mathbb A}}    
\nc{\bbB}{{\mathbb B}}    
\nc{\bbC}{{\mathbb C}}    
\nc{\bbD}{{\mathbb D}}    
\nc{\bbE}{{\mathbb E}}    
\nc{\bbF}{{\mathbb F}}    
\nc{\bbG}{{\mathbb G}}    
\nc{\bbH}{{\mathbb H}}    
\nc{\bbI}{{\mathbb I}}    
\nc{\bbJ}{{\mathbb J}}    
\nc{\bbK}{{\mathbb K}}    
\nc{\bbL}{{\mathbb L}}    
\nc{\bbM}{{\mathbb M}}    
\nc{\bbN}{{\mathbb N}}    
\nc{\bbO}{{\mathbb O}}    
\nc{\bbP}{{\mathbb P}}    
\nc{\bbQ}{{\mathbb Q}}    
\nc{\bbR}{{\mathbb R}}    
\nc{\bbS}{{\mathbb S}}    
\nc{\bbT}{{\mathbb T}}    
\nc{\bbU}{{\mathbb U}}    
\nc{\bbV}{{\mathbb V}}    
\nc{\bbW}{{\mathbb W}}    
\nc{\bbX}{{\mathbb X}}    
\nc{\bbY}{{\mathbb Y}}    
\nc{\bbZ}{{\mathbb Z}}    

\makeatletter
\ifdefined\p@enumi{
\renewcommand{\p@enumi}{A}
}
\makeatother

\nc{\orthogonal}{O(n)}  

\newcommand{\cyan}[1]{{\color{black}{#1}}} 
\newcommand{\morim}[1]{{\color{black}{#1}}} 
\newcommand{\kk}[1]{{\color{black}{#1}}} 
\newcommand{\res}[1]{{\color{black}{#1}}} 
\newcommand{\rev}[1]{{\color{black}{#1}}} 

\title{Minimum energy density steering of linear systems with Gromov-Wasserstein terminal cost}
\author{Kohei Morimoto and Kenji Kashima
\thanks{The authors are with the Graduate School of Informatics, Kyoto University,
	Kyoto, Japan
        {\tt\small \res{kohei.morimoto.73r@st.kyoto-u.ac.jp}; kk@i.kyoto-u.ac.jp}.
        This work was supported by JSPS KAKENHI Grant Number JP21H04875 and
the joint project of Kyoto University and Toyota Motor
Corporation, titled ``Advanced Mathematical Science for
Mobility Society".}%
}

\begin{document}

\maketitle
\thispagestyle{empty}
\pagestyle{empty}

\begin{abstract}

\rev{In this paper, we newly formulate and solve the optimal density control problem with Gromov-Wasserstein (GW) terminal cost in discrete-time linear Gaussian systems. Differently from the Wasserstein or Kullback-Leibler distances employed in the existing works, the GW distance quantifies the difference in \emph{shapes} of the distribution, which is invariant under translation and rotation. Consequently, our formulation allows us to find small energy inputs that achieve the desired shape of the terminal distribution, which has practical applications, e.g., robotic swarms.  
}
We demonstrate that the problem can be reduced to a Difference of Convex (DC) programming, which is efficiently solvable through the DC algorithm. 
Through numerical experiments, we confirm that \rev{the state distribution reaches the terminal distribution that can be realized with the minimum control energy among those having the specified shape.}

\end{abstract}


\section{Introduction}

Optimal density control is defined as the problem of controlling the probability distribution of state variables to the desired distribution in a dynamic system.
Promising applications of optimal density control include systems in which it is important to manage errors in the state, such as quality control and aircraft control, as well as quantum systems in which the distribution of the state itself is the object of control \cite{chen_optimal_2016}.

\rev{The problem addressed in this paper is a variant of the (finite-time) covariance steering problem for discrete-time linear Gaussian systems. Among the long history of this line of research \cite{anderson_inverse_1969, hotz_covariance_1985}, the most related recent works are as follows: 
}
hard constraint formulations, as seen in \cite{liu2022optimal, rapakoulias2023discrete, ito_2023_lq, ito_2023_schrodinger}, where the terminal state distribution is enforced as a constraint; and soft constraint formulations, as seen in \cite{halder_finite_2016, balci_covariance_2021, balci_exact_2022}, where the Wasserstein distance between the terminal distribution and the target distribution is incorporated as a cost.
In particular, Balci et al. \cite{balci_exact_2022} presented an optimal distributional control problem for discrete-time linear Gaussian systems using Wasserstein distance as the terminal cost, formulated as semidefinite programming (SDP), a form of convex programming, to derive globally optimal control policies.

\rev{To motivate the present work, let us consider the state distribution as an ensemble of particles or a multi-robotic swarm \cite{KM}. In such applications, a particular \emph{shape} of the formation is required to be achieved, but its location and orientation are often irrelevant. For example, they may seek to align in a single row in a two-dimensional region (See Fig.~\ref{fig:traj} below), or only the configuration may be specified based on inter-agent distance \cite{dimarogonas2008stability}.
    The aforementioned formulations can address the realization of the configuration with a fixed orientation but cannot address the optimization with respect to the rotation. 
}
To tackle this issue, we propose a novel density control problem incorporating the Gromov-Wasserstein distance (GW distance) as the terminal cost \cite{memoli2011gromov}. 
The GW distance is the distance between probability distributions and can measure the closeness of the \rev{shape} of the probability distributions.
By integrating the GW distance between the state and target distributions into the terminal cost, we can formulate the problem of controlling the \rev{shape} of the state distribution.
\rev{This problem can be viewed as a simultaneous optimization of the dynamical steering and the rotation of the target shape, which clearly contrasts the existing formulations. }

In this study, we focus on scenarios where the initial and target distributions are Gaussian and seek the optimal control policy among linear feedback control laws. 
While computing the GW distance between arbitrary distributions is challenging, it has recently been shown that the Gaussian Gromov-Wasserstein (GGW) distance, which is a relaxation of the GW distance for normal distributions, can be easily calculated \cite{delon_gromovwasserstein_2022}.
We show that the optimal density control problem with GW terminal cost can be formulated as a difference of convex (DC) programming problem. 
We solve the problem by the DC algorithm (DCA) \cite{tao1997convex}, a technique for solving DC programming problems through iterative convex relaxation. 
Remarkably, the convexified problem is transformed into a SDP form, which can be efficiently solved using standard convex programming solvers.

The rest of the paper is organized as follows.
In Section \ref{sec:problem setting}, we introduce the concept of the GW distance and present the optimal density steering problem with the GW distance as the terminal cost. 
Section \ref{sec:dc_programming} discusses the formulation of the problem as a DC programming problem, highlighting the objective function's nature as a difference of convex function and deriving the convexified sub-problem used in the DC algorithm. 
The numerical simulations are presented in Section \ref{sec:experiment}.
Finally, we conclude our paper in Section \ref{sec:conclusion}.

\noindent {\bf Notation}\hspace{5mm}
Let $\semidefinite{n}$ and $\definite{n}$ denote the sets of $n$-dimensional positive semidefinite matrices and positive definite matrices, respectively.
Let $\orthogonal$ denote the set of $n$-dimensional orthogonal matrices.
For matrices, \cyan{$\ffnmlr{\cdot}$ denotes the Frobenius norm.}
For a convex function $f$, $\partial f(x)$ denotes the set of subgradients of $f$ at $x$.
Let $\mathcal{P}(\mathcal{X})$ denote the set of all probability distributions over $\mathcal{X}$.
Let $\normal_n(\mu, \Sigma)$ denote the multivariate normal distribution with mean $\mu \in \R^n$ and covariance $\Sigma \in \cyan{\semidefinite{n}}$.
Let $\mathscr{N}_n$ denote the set of all $n$-dimensional multivariate normal distributions.

\section{Problem Setting}\label{sec:problem setting}
\subsection{Gromov-Wasserstein distance}

The optimal transport distance is generally defined as the minimized transport cost of transporting one probability distribution to another probability distribution.
The GW distance is the distance between probability distributions and is a variant of optimal transport distance, similar to the Wasserstein distance.
\cyan{Given two metric spaces $\mathcal{X}, \mathcal{Y}$,} the set of transports $\Pi$ between probability distributions $\mu \in \mathcal{P}(\mathcal{X})$ and $\nu \in \mathcal{P}(\mathcal{Y})$ is defined by
\begin{multline}
\Pi(\mu, \nu) := \{\pi(x, y) | \int_x \pi(x, y) dx = \nu(y), \\ \int_y \pi(x, y) dy = \mu(x)\}.
\end{multline}
Each element $\pi(x, y)$ in $\Pi(\mu, \nu)$ represents how the weight $\mu(x)$ at $x$ is transported to $y$, with the condition that $\int_x \pi(x, y) dx = \nu(y)$ ensuring that the transported destination becomes $\nu(y)$.
The GW distance is defined by
\begin{multline}
GW^2(\mu, \nu) \defeq \\ \inf_{ \pi\in\Pi(\mu, \nu)} \int \int \lr{\nmlr{x-x'}_{\mathcal{X}} - \nmlr{y-y'}_{\mathcal{Y}}}^2 \label{eq:GW}\\
\pi(x, y)\pi(x', y')dx dy dx' dy',
\end{multline}
where $\nmlr{x-x'}_{\mathcal{X}}$ and $\nmlr{y-y'}_{\mathcal{Y}}$ represent the norms in the spaces $\mathcal{X}$ and $\mathcal{Y}$, respectively.
The GW distance is small when points that are close (resp.~farther apart) before transportation are brought closer together (resp.~farther apart) after transportation. Conversely, the GW distance increases when points that were initially close are moved farther apart after transportation.
\rev{Therefore, this definition quantifies the shape difference between two probabilistic distributions.
For comparison, recall that the Wasserstein distance is defined as 
\begin{equation}\label{eq:Wasserstein}
W^2(\mu, \nu) \defeq \inf_{ \pi\in\Pi(\mu, \nu)} \int \int d(x,y)^2
\pi(x, y)dx dy,
\end{equation}
where $\mathcal{X}=\mathcal{Y}$ and $d(\cdot,\cdot)$ is a suitable distance on $\mathcal{X}$.  
While the Wasserstein distance is sensitive to the absolute positions or orientations of the distributions, the GW distance is invariant under isometric transformations such as translations and rotations.}

The GW distance involves a non-convex quadratic program over transport $\pi$, making it challenging to compute the GW distance between arbitrary probability distributions.
Recently, it has been shown that the Gaussian Gromov-Wasserstein (GGW) distance, where the transport is constrained to be Gaussian distribution only, can be explicitly expressed in terms of the parameters of the normal distributions \cite{delon_gromovwasserstein_2022}. 
Specifically, the GGW between Gaussian distributions $\mu \in \mathscr{N}_m$ and $\nu \in \mathscr{N}_n$ is defined by
\begin{multline}
GGW^2(\mu, \nu) \defeq \\ \inf_{\pi\in\Pi(\mu, \nu) \cap \mathscr{N}_{m+n} } \int \int \lr{\nmlr{x-x'} - \nmlr{y-y'}}^2 \\ \pi\lr{x, y}\pi(x', y')dx dy dx' dy'
\end{multline}
where $\Pi(\mu, \nu) \cap \mathscr{N}_{m+n}$ represents the restriction of the transport to the ($m+n$)-dimensional Gaussian distribution $\mathscr{N}_{m+n}$.
For $\mu = \normal_m(\mu_0, \Sigma_0), \nu = \normal_n(\mu_1, \Sigma_1)$, it holds that
\begin{multline}
\label{eq:ggw}
GGW^2(\mu, \nu) = \\ 4\lr{\tr(\Sigma_0) - \tr(\Sigma_1)}^2 + 8\ffnmlr{D_0 -D_1}^2,
\end{multline}
where $D_0, D_1$ are the diagonal matrices with the eigenvalues of $\Sigma_0, \Sigma_1$ sorted in descending order, and if $m \neq n$, the missing elements are filled with zeros.

\subsection{Optimal density steering with Gromov-Wasserstein terminal cost}

Let $n_x$ be the dimension of the state space and $n_u$ the dimension of the input, and consider the following discrete-time linear Gaussian system.
\begin{subequations}
    \begin{align}
        \label{eq:dynamics}
        x_{k+1} &= Ax_k + Bu_k + w_k \\
        x_0 &= \normal(0, \Sigma_0) \\
        w_k &\sim \normal \lr{0, W_k}
    \end{align}
\end{subequations}
Here, the covariance matrix of the initial Gaussian distribution is $\Sigma_0 \in \definite{n_x}$ and that of the noise is $W_k \in \semidefinite{n_x}$. 
For control input $u_k$, We use a stochastic linear control policy as
\begin{align}
\label{eq:policy}
u_k(x) = \normal(K_k x, Q_k),
\end{align}
where $K_k$ is feedback gain and $Q_k \succeq O$ is covariance of Gaussian distribution.
We consider the problem of minimizing the sum of the control costs and the Gromov-Wasserstein distance between the terminal distribution $\rho_{N} = \normal(\mu_N, \Sigma_N)$ ($\Sigma_N \in \definite{n_x}$) and the target distribution $\rho_r = \normal(0, \Sigma_r)$ ($\Sigma_r \in \semidefinite{n_x}$). 
The objective function is represented as 
\begin{subequations}
\begin{align}
\min_{K_k, Q_k} & J(K_k, Q_k) \\
J(K_k, Q_k) &= \lambda \ex \left[ \sum_{k=0}^{N-1} {u_k^T R_k u_k} \right] + GGW^2(\rho_N, \rho_r), \label{eq:objective}
\end{align}
\end{subequations}
where $R_k \in \definite{n_x}$ denotes the weights for control cost.
Using the control policy \eqref{eq:policy} in system \eqref{eq:dynamics}, the probability distribution of the state $x_N$ at the terminal time $N$ will also be the Gaussian distribution. Thus, by substituting equations \eqref{eq:ggw} and \eqref{eq:policy} into equation \eqref{eq:objective}, we obtain
\begin{dmath}
\label{eq:objective_assign}
J(K_k, Q_k) = \lambda \sum_{k=0}^{N-1} {\tr\lr{R_k \lr{ K_k\Sigma_k K_k^T + Q_k}}} + 4\lr{ \tr(\Sigma_N) - \tr(\Sigma_r)}^2 + 8\ffnmlr{\Sigma_N}^2 - 16 \tr{(D_ND_r)},
\end{dmath}
where $\Sigma_k$ is the covariance matrix of the state $x_k$, and $D_N$, $D_r$ are diagonal matrices with the eigenvalues of $\Sigma_N$, $\Sigma_r$ arranged in descending order. 
The dynamics of $\Sigma_k$ is given by
\morim{
\begin{dmath}
\label{eq:covar_dynamics}
    \Sigma_{k+1} = A\Sigma_k A^T + B K_k \Sigma_k A^T + A \Sigma_k K_k^T B^T + B K_k \Sigma_k K_k^T B^T + BQ_k B^T + W_k.
\end{dmath}
}
Here, we introduce the variable transformations \morim{$M_k := P_k\Sigma_k^{-1} P_k^T+Q_k$} and $P_k := K_k \Sigma_k$ as in the Ref.~\cite{chen_optimal_2016-1, balci_exact_2022}.
To guarantee the invertibility of the variable transformation, we need $M_k \succeq O$ and \morim{$M_k - P_k \Sigma_k^{-1} P_k^T \succeq O$}, which implies that the following condition must be satisfied:
\begin{align}
\label{eq:cond_q_positive}
\begin{bmatrix}
M_k & P_k \\
P_k^T & \Sigma_k
\end{bmatrix} \succeq O.
\end{align}
This condition is added to the optimization problem to ensure the feasibility of the solution.
Finally, from the \eqref{eq:objective_assign}, \eqref{eq:covar_dynamics} and \eqref{eq:cond_q_positive}, we can write the optimization problem to be solved as follows:

\begin{subequations}
\label{eq:problem}
\begin{align}
    \min_{\substack{\Sigma_k, M_k, P_k  }} &  \quad  J(\Sigma_N, M_k)\\
    \begin{split}
        J(\Sigma_N, M_k) &= \lambda \sum_{k=0}^{N-1} \tr\lr{R_k M_k} \\ &+  4\lr{\tr(\Sigma_N) - \tr(\Sigma_r)}^2 \\ &+ 8\ffnmlr{\Sigma_N}^2 - 16 \tr{(D_ND_r)} 
    \end{split} \label{eq:J}\\
    \begin{split}
    \mathrm{s.t.} \quad 
    & 
    \Sigma_{k+1} = A_k\Sigma_k A_k^T  + A_k P_k^T B_k^T \\ & + B_k P_k A_k^T + B_k M_k B_k^T + W_k
    \end{split} \label{eq:covar_constrint}
    \\
    & \begin{bmatrix}
    M_k & P_k \\
    P_k^T & \Sigma_k  
  \end{bmatrix} \succeq O \label{eq:sdp_constrint}
\end{align}
\end{subequations}

Although we assumed a stochastic strategy as a control law in \eqref{eq:policy}, the optimal solution turns out to be a deterministic strategy.
\begin{theorem}
\label{theorem:deterministic}
Suppose $\{A_k\}_{k=0}^{N-1}$ are invertible. Then, the optimal policy in problem \eqref{eq:problem} is deterministic, that is, the optimal solution satisfies $Q_k = M_k - K_k\Sigma_k K_k^T = O$.
\end{theorem}
\begin{proof}
The proof is provided using the same arguments in the Ref.~\cite{balci_exact_2022}.
We utilize the Karush-Kuhn-Tucker conditions. 
Let $E_k$ denote the Lagrange multiplier for constraint \eqref{eq:covar_constrint}, and let $F_k$ denote the Lagrange multiplier for constraint \eqref{eq:sdp_constrint}, represented as
\begin{align*}
    F_k = \begin{bmatrix}
    F_k^{00} & F_k^{01} \\
    F_k^{10} & F_k^{11}  
  \end{bmatrix}.
\end{align*}
From the stationarity condition, we obtain
\begin{align}
    & B_k^T E_k A_k + F_k^{01} = 0 \label{eq:proof1}\\
    & R_k - B_k^T E_k B_k + F_k^{00} = 0. \label{eq:proof2}
\end{align}
The complementary slackness condition yields
\begin{align}
    \begin{bmatrix}
    F_k^{00} & F_k^{01} \\
    F_k^{10} & F_k^{11}  
  \end{bmatrix}
\begin{bmatrix}
    M_k & P_k \\
    P_k^T & \Sigma_k
\end{bmatrix} = O,
\end{align}
implying
\begin{align*}
    \begin{bmatrix}
    F_k^{00} & F_k^{01} \\
    F_k^{10} & F_k^{11}  
  \end{bmatrix}
\begin{bmatrix}
    I & P_k\Sigma_k^{-1} \\
    O & I
\end{bmatrix}
\begin{bmatrix}
    M_k - P_k\Sigma_k^{-1}P_k^T & O \\
    O & \Sigma_k
\end{bmatrix} = O
\end{align*}
from the definiteness of $\Sigma_k$.
Thus, we obtain
\begin{align}
    & F_{00} \lr{M_k - P_k \Sigma_k^{-1}P_k^T} = O \label{eq:proof3}\\
    & F_{01}^T \lr{M_k - P_k \Sigma_k^{-1}P_k^T} = O. \label{eq:proof4}
\end{align}
Subsequently, by combining \eqref{eq:proof1}, \eqref{eq:proof4}, and the invertibility of $A_k$, we derive
\begin{align}
    \label{eq:proof5}
    E_k B_k \lr{M_k - P_k \cyan{\Sigma_k^{-1}} P_k^T} = O.
\end{align}
Finally, from \eqref{eq:proof2}, \eqref{eq:proof3}, and \eqref{eq:proof5}, we conclude that 
\begin{align*}
    &(R_k - B_k^T E_k B_k) (M_k - P_k \Sigma_k^{-1} P_k^T)\\
    &=  R_k (M_k - P_k \Sigma_k^{-1} P_k^T) \\
    &= O.
\end{align*}
and due to the positive definiteness of $R_k$, we have $M_k - P_k \Sigma_k^{-1} P_k^T = O$.
\end{proof}

\section{Formulation as Difference of Convex Programming}

\begin{figure}
    \centering
    \includegraphics[scale=0.35]{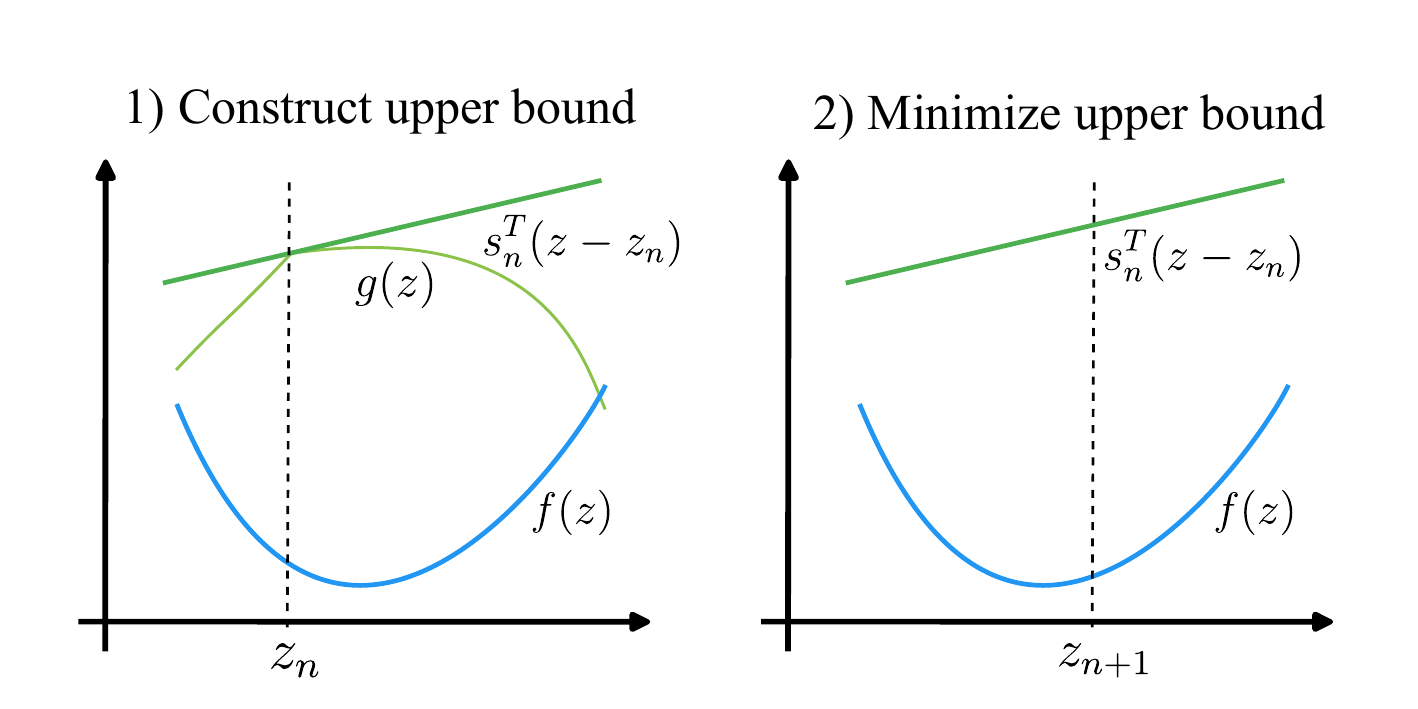}
    \caption{Visualization for the algorithm of DCA.} \label{fig:dcalgorithm}
\end{figure}
\label{sec:dc_programming}
In this section, we show that problem \eqref{eq:problem} is a DC programming problem and solve it using the DCA, an optimization method for DC programming. 
The DC programming problem is an optimization problem whose objective function is a DC function, which is expressed as the difference between two convex functions.
Since the DC programming problem is a non-convex optimization problem, finding a global optimum is generally challenging.
However, several optimization methods that efficiently find solutions by exploiting the properties of DC functions have been proposed. 
These include global optimization techniques using branch and bound methods \cite{horst1999dc}, and methods for finding sub-optimal solutions, such as the DCA \cite{tao1997convex} and the Concave-Convex Procedure (CCCP) \cite{yuille2001concave}.

\cyan{The DCA iteratively constructs a convex upper bound for the objective function, minimizes this upper bound, and then updates the upper bound using the minimizer of the previous iteration.}
Assume the objective function $h$ to be optimized is expressed as $h(z) = f(z) - g(z)$ by convex functions $f(z)$ and $g(z)$ defined on a convex set $\Omega$. 
The DCA iterates the following steps until convergence:
\begin{enumerate}
\item Construct the upper bound $\hat{h}(z)$ of $h(z)$ as
\begin{align*}
    \hat{h}(z) = f(z) - s_n^T (z - z_n),
\end{align*}
 where $s_n \in \partial g(z_n)$. 
\item Set $z_{n+1} = \min_{z\in \Omega} \hat{h}(z)$.
\end{enumerate}
\cyan{Because $\hat{h}(z)$ is a convex function over the convex set, we are able to minimize this convex subproblem efficiently.}
It is known \cite{tao1997convex} that when the optimal value of the problem is finite and the sequences ${z_{n}}$, ${s_{n}}$ are bounded, any accumulation points $z^{\infty}$ of ${z_{n}}$ are critical points of $f-g$, which implies $0 \in \partial(f - g)(z^\infty)$.
In Figure \ref{fig:dcalgorithm}, we show the visualization of the DCA algorithm.

In the next proposition and theorem, we show that problem \eqref{eq:problem} is a DC programming problem.
\begin{proposition}[Anstreicher and Wolkowicz\cite{anstreicher2000lagrangian}]
\label{prop:anstreicher}
Let $A$ and $B$ be $n \times n$ symmetric matrices decomposed into their eigenvalues as $A = V \Lambda V^T$ and $B = W \Xi W^T$, respectively. Assume the eigenvalues in $W$ and $V$ are arranged in descending order. Then,
\begin{align*}
\max_{U\in \orthogonal} \tr(UAU^TB)
\end{align*}
has an optimal solution $U^* = WV^T$, and the optimal value is $\tr(\Lambda\Xi)$.
\end{proposition}

\begin{theorem}
$J(\Sigma_N, M_k)$ in \eqref{eq:J} is a DC function.
\end{theorem}
\begin{proof}
Since $\lambda \sum_{k=0}^{N-1} (\tr(R_k M_k)) + 4(\tr(\Sigma_N) - \tr(\Sigma_r))^2 + 8\ffnmlr{\Sigma_N}^2$ is clearly a convex function, it suffices to show that 
\kk{
\begin{align}\label{eq:g}
    g(\Sigma_N):=\tr(D_ND_r)
\end{align}
}
is a convex function. From Proposition \ref{prop:anstreicher}, we have
\begin{align*}
\tr(D_ND_r) = \max_{U\in \orthogonal} \tr(U\Sigma_NU^T\Sigma_r).
\end{align*}
Thus, $\tr(D_ND_r)$ is the maximum of linear functions, making it a convex function with respect to $\Sigma_N$. 
Precisely, let us consider a scalar $\alpha \in [0, 1]$ \cyan{and} two positive definite matrices $\Sigma$ and $\Sigma'$. Then, \cyan{the convex combination \morim{$\alpha \Sigma + (1-\alpha) \Sigma'$} satisfies the following inequality:}
\begin{align*}
&\max_{U\in \orthogonal} \tr(U\lr{\alpha\Sigma+(1-\cyan{\alpha})\Sigma'}U^T\Sigma_r) \\
&\leq \max_{U\in \orthogonal} \alpha\tr(U\Sigma U^T\Sigma_r) + (1-\alpha)\tr(U\Sigma' U^T\Sigma_r)\\
&\begin{multlined}
        \leq \alpha\max_{U\in \orthogonal}\tr(U\Sigma U^T\Sigma_r) \\+ (1-\alpha)\max_{U'\in \orthogonal}\tr(U'\Sigma' {U'}^T \Sigma_r). 
\end{multlined}
\end{align*}
\end{proof}

Furthermore, in the next theorem, we derive a subgradient of the concave part of $J(M_k, \Sigma_k)$ to construct the upper bound in DCA. 
\begin{theorem}
\label{theorem:subgradient}
Assume the eigenvalue decompositions: $\Sigma_N = V_N D_N V_N^T$ and $\Sigma_r$ is $\Sigma_r = V_r D_r V_r^T$. Then, 
$V_N D_r V_N^T \in \semidefinite{n_x}$
is a subgradient of $g(\Sigma_N)$ in \eqref{eq:g}.
\end{theorem}
\begin{proof}
From Proposition \ref{prop:anstreicher}, we have
\begin{align*}
U^* := \argmax_{U\in \orthogonal} \tr(U\Sigma_NU^T\Sigma_r) = V_r V_N^T.
\end{align*}
Therefore, by Danskin's theorem\cite{bertsekas1997nonlinear}, a subgradient can be obtained by differentiating the function inside the max operation with respect to $\Sigma_N$ and then substituting $U^*$. Hence, 
\begin{align*}
{U^*}^T\Sigma_rU^* = V_N V_r^T \Sigma_r V_r V_N^T = V_N D_r V_N^T \in \partial \kk{g}(\Sigma_N).
\end{align*}
\end{proof}
Therefore, the convex subproblem in DCA is formulated as
\begin{subequations}
\begin{align}
    \label{eq:partial_problem}
    \begin{split}
        \min_{\substack{\Sigma_k, M_k, P_k  }} \quad & \lambda  \sum_{k=0}^{N-1} \tr\lr{R_k M_k} + \\ &4\lr{\tr(\Sigma_N) - \tr(\Sigma_r)}^2 \\+ &8\ffnmlr{\Sigma_N}^2 - 16 \tr{(\Sigma_N  {V_N^{(n)T}} \cyan{D_r} V_N^{(n)})} 
    \end{split} \\
    \begin{split}
    \mathrm{s.t.} \quad 
    & 
    \Sigma_{k+1} = A_k\Sigma_k A_k^T  + A_k P_k^T B_k^T \\ & + B_k P_k A_k^T + B_k M_k B_k^T + W_k
    \end{split}
    \\
    & \begin{bmatrix}
    M_k & P_k \\
    P_k^T & \Sigma_k  
  \end{bmatrix} \succeq 0 
\end{align}
\end{subequations}
where $V_N^{(n)}$ is the matrix obtained by decomposing the optimal $\Sigma_N$ in the $n$-th iteration of DCA. The term $\tr{(\Sigma_N  V_N^{(n)T}\cyan{D_r}V_N^{(n)})}$ represents linear lower bound of convex function $l(\Sigma_N)$ using a subgradient obtained in Theorem \ref{theorem:subgradient}.
The subproblem is a semidefinite programming problem (SDP), which can be efficiently solved. 
We use the solution from each optimization step to update the value of $V_N$. 
By iteratively applying the optimization process, the solution progressively approaches a sub-optimal solution for the original problem \eqref{eq:problem}.

In Theorem \ref{theorem:deterministic}, we showed that the optimal policy for Problem (\ref{eq:problem}) is deterministic. 
\kk{The following theorem shows that the proposed algorithm generates a sequence of deterministic control policies:}
\begin{theorem}
    When $\{A_k\}_{k=0}^{N-1}$ are invertible, the optimal policy of the subproblem \eqref{eq:partial_problem} is also deterministic.
\end{theorem}
\begin{proof}
    The KKT conditions used in the proof of Theorem \ref{theorem:deterministic} also hold in the subproblem \eqref{eq:partial_problem}.
\end{proof}

\section{Numerical Experiments}
\label{sec:experiment}

In this section, we perform numerical optimization for problem \eqref{eq:problem} using DCA.
We set the parameters in this experiment as 
\begin{align*}
A_k = \begin{bmatrix}
1.0 & 0.1 \\
-0.3 & 1.0 \\
\end{bmatrix},~
&B_k = \begin{bmatrix}
0.7 \\
0.4 \\
\end{bmatrix} \\

\Sigma_0 = \begin{bmatrix}
3 & 0 \\
0 & 3 \\
\end{bmatrix},~

& W_k = 0.5I_2,\\
R_k = 1.0,~&N = 10.
\end{align*}
\kk{Figure~\ref{fig:autonomous} shows the time evolution of state covariance of the uncontrolled system. }
For the implementation of the convex subproblem in DCA, we used the MOSEK solver\cite{mosek} and the CVXPY modeler\cite{diamond2016cvxpy}.

\begin{figure}
    \centering
    \includegraphics[scale=0.36]{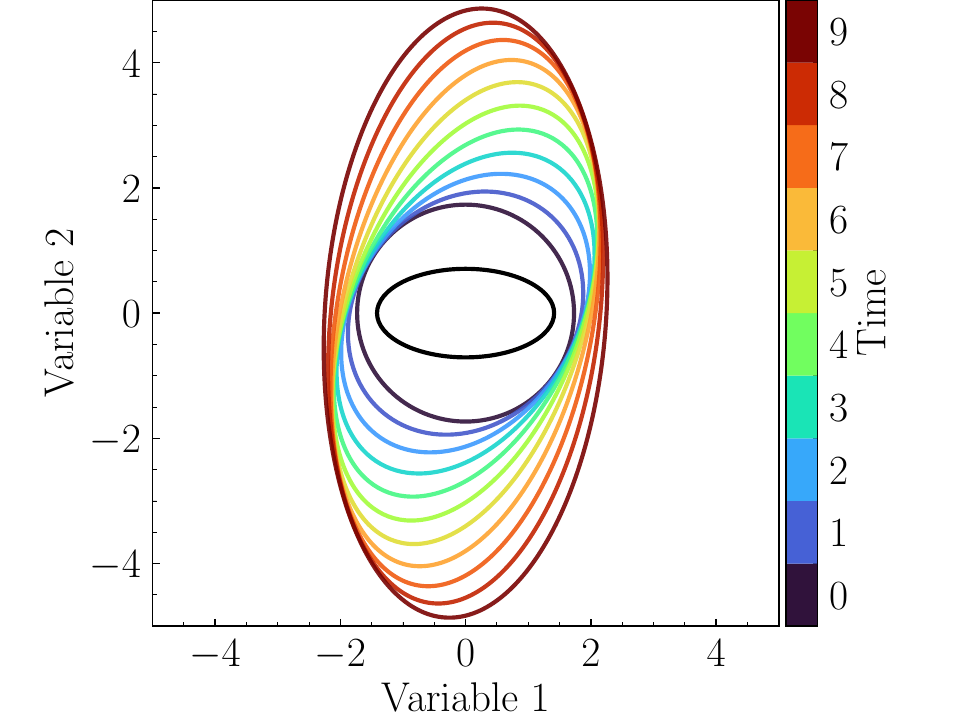}
    \caption{The snapshots of the covariance matrices of uncontrolled system.
    The colored lines indicate the $2\sigma$ range of $\Sigma_k$, while the black line represents $\Sigma_r$ in \eqref{eq:Sigma_r1}.
} \label{fig:autonomous}
\end{figure}

\subsubsection{Line alignment}
First, we consider the case where the desired density is Gaussian $\rho_r=\normal(0,10)$, which is not on $\mathbb{R}^2$, but on $\mathbb{R}$. The problem seeks the optimal policy to align the terminal distribution into one line.
\rev{Note that $W(\rho_N,\rho_r)$ in \eqref{eq:Wasserstein} does not make sense\footnote{\rev{One may think we can embed it onto $\R^2$ (e.g., by \eqref{eq:Sigma_r2}) and consider the Wasserstein distance. However, there is a rotational degree of freedom, which affects the resulting distance. We can interpret that our formulation optimizes this rotation in the sense of the required control energy; See Fig.~\ref{fig:traj}. } }
because $\mathcal{X}\neq \mathcal{Y}$. 
In contrast, the GW distance $GW(\rho_N,\rho_r)$ in \eqref{eq:GW} is well-defined and,}
thanks to \eqref{eq:ggw}, equivalent to take
\begin{align}\label{eq:Sigma_r2}
    \Sigma_r = \begin{bmatrix}
    10 & 0 \\
    0 & 0 \\
    \end{bmatrix}.
\end{align}
Figure \ref{fig:traj} presents the trajectories of one hundred samples from the controlled process when the target distribution is degenerate distribution.
It can be seen that the distribution of states actually stretches vertically to \rev{achieve a one-line alignment}.

\begin{figure}
    \centering
    \includegraphics[scale=0.215]{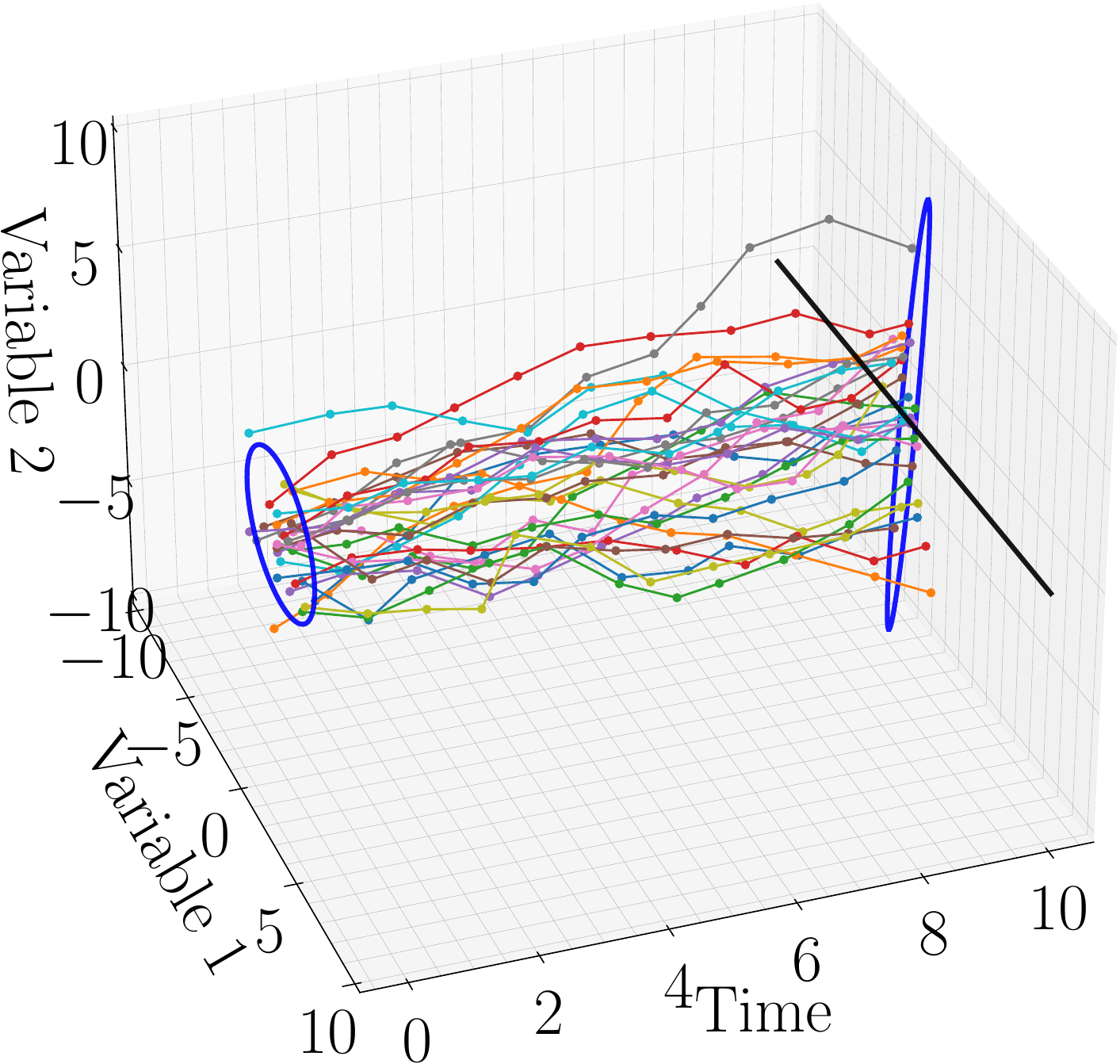}
    \caption{One hundred sample paths of the systems controlled by the optimal policy for $\lambda = 1$. 
The blue circle denotes 2$\sigma$ range of $\Sigma_0$ and $\Sigma_N$ and the black \kk{line} denotes the degenerated target distribution $\Sigma_r$ in \eqref{eq:Sigma_r2}.
} \label{fig:traj}
\end{figure}

\subsubsection{Comparison with the Wasserstein formulation}
Let us consider 
\begin{align}\label{eq:Sigma_r1}
    \Sigma_r = \begin{bmatrix}
    2 & 0 \\
    0 & 0.5 \\
    \end{bmatrix}
\end{align}
for which 
\begin{align}\label{eq:GGWuncontrolled}
    GGW^2(\rho_N,\rho_r)=6711.44
\end{align}
for the uncontrolled system.
Figure \ref{fig:result_sample} shows the snapshots of state covariance under the optimal control input obtained by DCA. 
It can be observed that the shape of the terminal distribution approaches that of the target distribution as $\lambda$ decreases.
\cyan{As shown below,} the terminal distribution \cyan{is} the one requiring the least energy among the rotated distributions of the target due to the rotational invariance of the GW distance.

Figure \ref{fig:costs_lambda} shows the relationship between the optimized control cost term and GW cost term in Eq. \eqref{eq:objective} for each $\lambda$.
As the value of $\lambda$ increases, the control cost rises, while the GW cost decreases. 
Conversely, as the value of $\lambda$ decreases, the control cost diminishes, and the GW cost increases.
Also, as $\lambda$ becomes smaller, the GW cost is almost close to zero.
It is noteworthy that, in comparison to the uncontrolled system in \eqref{eq:GGWuncontrolled}, our algorithm achieves a significant reduction in the GW cost.


\rev{Finally, we clarify the advantage of our approach over the Wasserstein terminal cost problem \cite{balci_exact_2022}.
In Fig.~\ref{fig:anim_1}, the obtained terminal distribution is $\rho_N\approx \mathcal{N}(0,\hat\Sigma_r(\theta_{\rm GW}))$ with $\theta_{\rm GW} = 1.20 \mathrm{\, [rad]}$ where $\hat\Sigma_r(\theta)$ is obtained by rotating $\Sigma_r$ by an angle $\theta$, i.e., 
\[
    \hat\Sigma_r(\theta) := R(\theta)^T \Sigma_r R(\theta),\ R(\theta):= \begin{bmatrix}
\cos\theta & -\sin\theta \\
\sin\theta & \cos\theta
\end{bmatrix}. 
\]
It is shown in \cite{balci_exact_2022} that we can solve
\begin{align}
& \min_{K_k, Q_k} \lambda \ex \left[ \sum_{k=0}^{N-1} {u_k^T R_k u_k} \right] + W^2(\rho_N, \mathcal{N}(0,\hat\Sigma_r(\theta)) \label{eq:Wass_objective}
\end{align}
by an SDP. 
Then, we solved this problem for a sufficiently small $\lambda$ (i.e., large terminal cost). 
The required control energy for the obtained optimal control input (i.e., the first term without $\lambda$ in \eqref{eq:Wass_objective}) is denoted by $W_{\rm opt}(\theta)$, which is shown in 
Fig.~\ref{fig:wasserstein_score}. It is noteworthy that the function exhibits a non-convexity.
We can also observe 
\[
    \theta_{\rm GW} \approx \theta^* := \argmin_\theta W_{\rm opt}(\theta), 
\]
which implies that the rotation angle obtained by the GW terminal cost problem minimizes the resulting control energy needed to realize the required shape (specified by $\Sigma_r$). 
From a computation cost point of view, while finding $\theta^*$ using the Wasserstein terminal cost approach requires performing optimization to compute $W_{\rm opt}(\theta)$ for each $\theta$, our GW terminal cost framework only necessitates solving a single optimization problem. 
Moreover, our approach remains computationally tractable even in high-dimensional settings, where the Wasserstein terminal cost approach becomes computationally intractable due to the exponential growth of the search space of the rotation matrix.
}

\begin{figure*}[thpb]
\centering
\begin{subfigure}[t]{0.33\textwidth}
    \centering
    \includegraphics[scale=0.33]{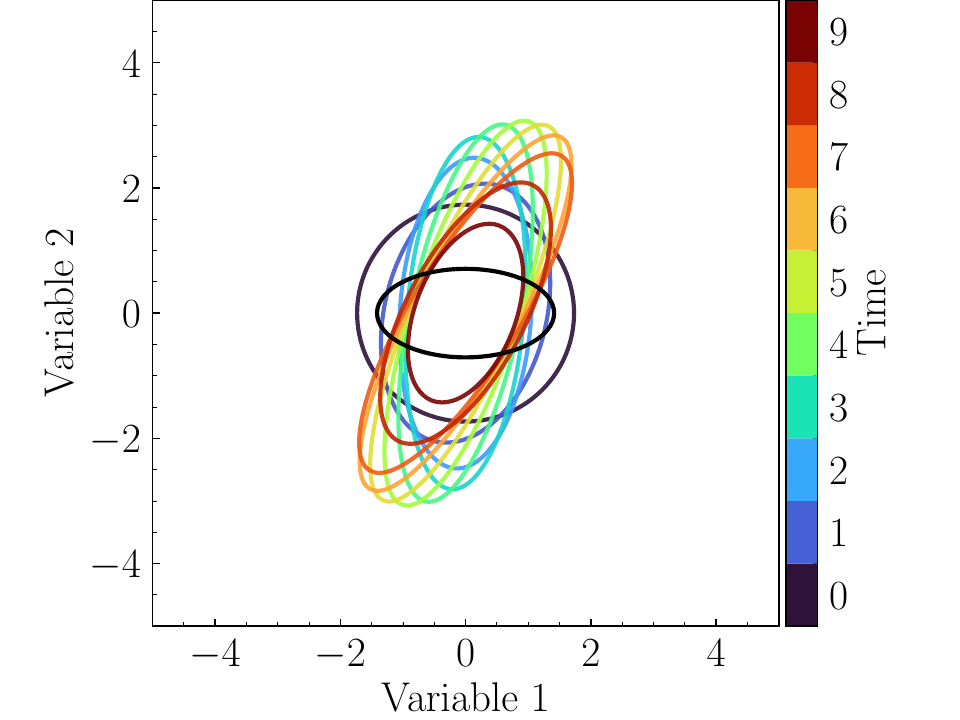}
    \caption{$\lambda = 1$} \label{fig:anim_1}
\end{subfigure}\hfill
\begin{subfigure}[t]{0.33\textwidth}
    \centering
    \includegraphics[scale=0.33]{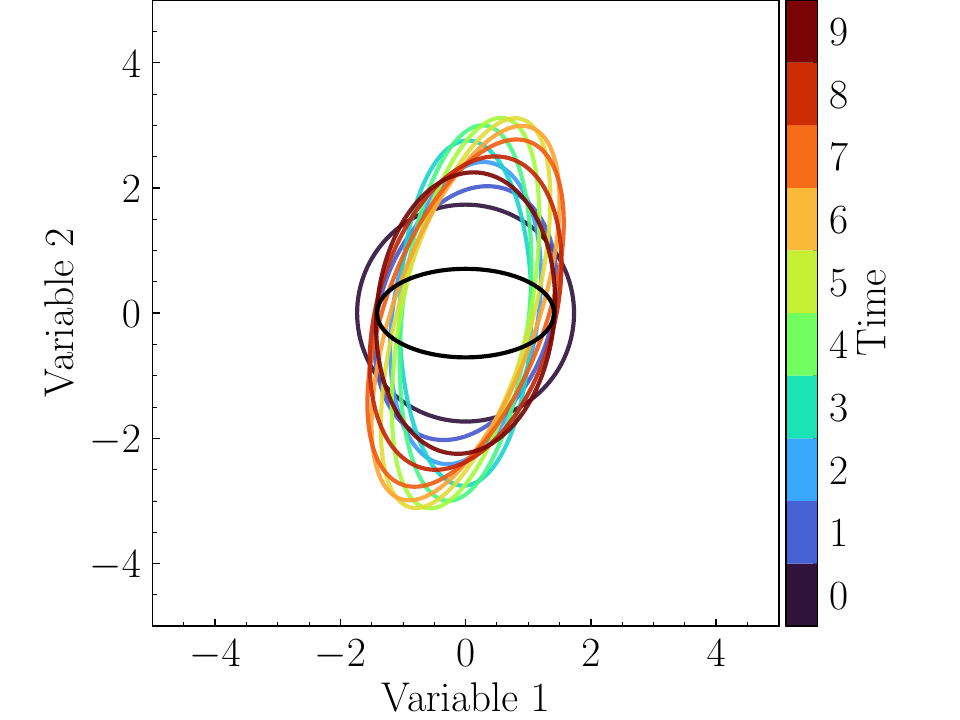}
    \caption{$\lambda = 100$} \label{fig:anim_100}
\end{subfigure}\hfill
\begin{subfigure}[t]{0.33\textwidth}
    \centering
    \includegraphics[scale=0.33]{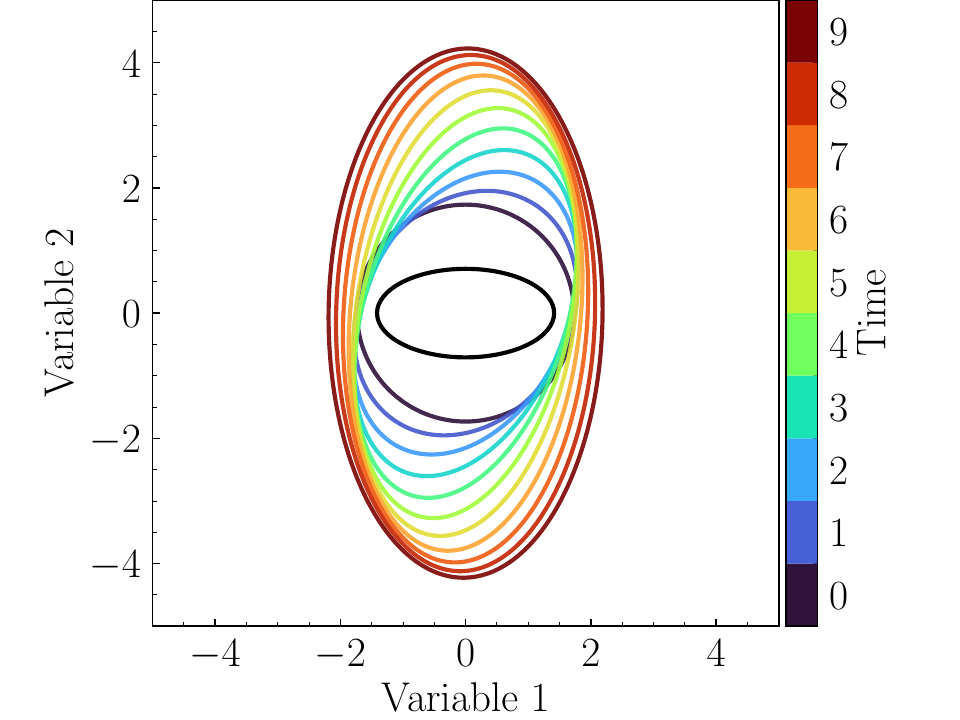}
    \caption{$\lambda = 10000$} \label{fig:anim_10000}
\end{subfigure}
\caption{The snapshots of the covariance matrices under the optimized policy for different values of $\lambda$. The colored lines indicate the $2\sigma$ range of $\Sigma_k$, while the black line represents $\Sigma_r$ in \eqref{eq:Sigma_r1}.}
\label{fig:result_sample}
\end{figure*}

\begin{figure}[t]
    \centering
    \includegraphics[scale=0.29]{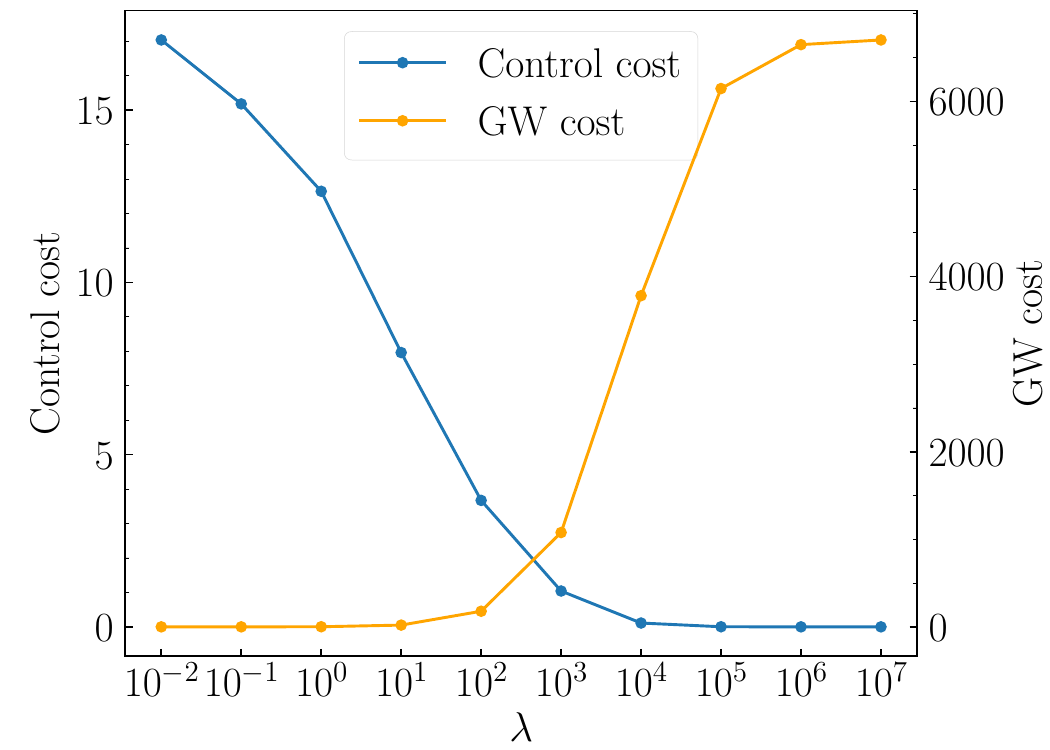}
    \caption{Relationship between the optimized control cost term and GW cost term in Eq. \eqref{eq:objective} for each $\lambda$.} \label{fig:costs_lambda}
\end{figure}

\begin{figure}
    \centering
    \includegraphics[scale=0.292]{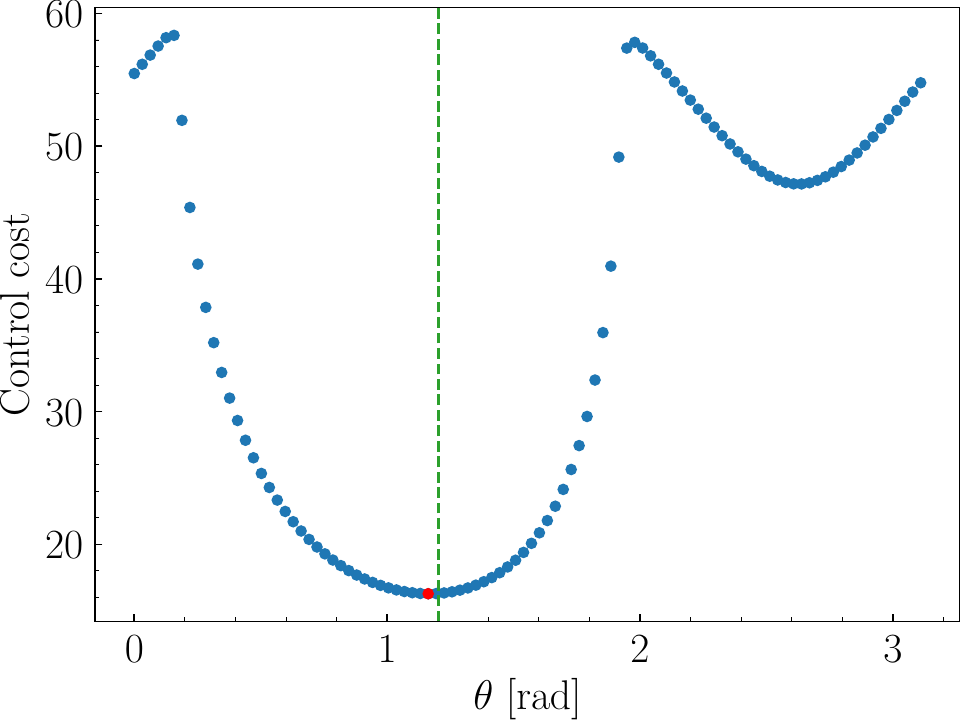}
    \caption{The required control energy $W_{\rm opt}(\theta)$ of the optimal input for the Wasserstein terminal cost problem in \eqref{eq:Wass_objective}. The green line and the red dot represent $\theta_{\rm GW}$ and
    $\theta^*$, respectively.  
} \label{fig:wasserstein_score}
\end{figure}

\section{Conclusion}
\label{sec:conclusion}
In this study, we addressed the optimal density control problem with the Gromov-Wasserstein distance as the terminal cost.
We showed that the problem is a DC programming problem and proposed an optimization method based on the DC algorithm.
Numerical experiments confirmed that \rev{the state distribution reaches the terminal distribution that can be realized with the minimum control energy among those having the specified shape.}

\rev{Future work includes the application of the proposed GW framework
to the transport between spaces equipped with different Riemannian metric structures or point clouds. }
\kk{Model predictive formation control based on a fast algorithm for optimal transport \cite{solomon_entropic_2016} is also a promising direction \cite{ITO2023110980}.
\rev{The convergence and computation complexity of the proposed DC algorithm should also be investigated.}
}

\bibliography{gromov}
\bibliographystyle{IEEEtran}

\end{document}